\documentclass[12pt, english, a4paper]{amsart}
\usepackage{nicematrix}

\usepackage[T1]{fontenc}
\usepackage[utf8]{inputenc}
\usepackage[english]{babel}

\usepackage{accents}
\usepackage{caption}
\usepackage{comment}
\usepackage{enumerate}
\usepackage[hmargin=2.5cm,vmargin=3.2cm]{geometry}
\usepackage{graphics}
\usepackage[colorlinks]{hyperref}
\hypersetup{allcolors = black}
\usepackage{url}
\usepackage{tikz}

\usepackage{nicematrix}

\usepackage{amsmath}
\usepackage{amssymb}
\usepackage{amsthm}
\usepackage{commath}
\usepackage{fourier}
\usepackage{mathtools}

\theoremstyle{plain}
\newtheorem{theorem}[]{Theorem}[section]
\newtheorem{proposition}[theorem]{Proposition}

\newtheorem{corollary}[theorem]{Corollary}

\numberwithin{equation}{section}

\newcommand{\R}{\mathbb{R}}

\newcommand{\dd}[1]{\,\mathrm{d}#1}
\newcommand{\mbf}[1]{\mathbf{#1}}
\newcommand{\one}[1]{\mathbf{1}_{#1}}
\newcommand{\zero}[1]{\mathbf{0}_{#1}}

\newcommand{\Vol}[2]{\mathrm{Vol}_{#1}\left(#2\right)}

\newcommand{\vv}{\mathbf{v}}

\newcommand{\1}{\mathbf{1}}

\DeclareMathOperator{\sinc}{sinc}

\DeclarePairedDelimiterX\sca[2]{\langle}{\rangle}{#1,#2}
\makeatletter
\let\oldsca\sca
\def\sca{\@ifstar{\oldsca}{\oldsca*}}
\makeatother

\usepackage{nicematrix}
\NiceMatrixOptions{xdots={line-style = <->}}



\begin{document}
\title{Estimates on the decay of the Laplace--Pólya integral}
\author{Gergely Ambrus \and Barnabás Gárgyán}

\thanks{\noindent
Research of G. A. was partially supported by the ERC Advanced Grant "GeoScape" no.  882971, by Hungarian National Research grants no. NKFIH K-147145 and NKFIH K-147544,  and by project no. TKP2021-NVA-09, which has been implemented with the support provided by the
Ministry of Innovation and Technology of Hungary from the National
Research, Development and Innovation Fund, financed under the
TKP2021-NVA funding scheme. This work was supported by the Ministry of Innovation and Technology NRDI Office within the framework of the Artificial Intelligence National Laboratory (RRF-2.3.1-21-2022-00004). 
Work of B. G. was supported by the ÚNKP-23-2-SZTE-584 New National Excellence Program of the Ministry for Culture and Innovation from the source of the National Research, Development and Innovation Fund, and by the Warwick Mathematics Institute Centre for Doctoral Training, and gratefully acknowledges funding by University of Warwick’s Chancellors' International Scholarship scheme.
}


\begin{abstract}
The Laplace--Pólya integral, defined by $J_n(r) = \frac1\pi\int_{-\infty}^\infty \sinc^n t \cos(rt) \dd t$, appears in several areas of mathematics. We study this quantity by combinatorial methods; accordingly, our investigation focuses on the values at integer $r$'s. Our main result establishes a lower bound for the ratio $\frac{J_n(r+2)}{J_n(r)}$ which extends and generalises the previous estimates of Lesieur and Nicolas~\cite{Lesieur-Nicolas}, and provides a natural counterpart to the upper estimate established in our previous work~\cite{Ambrus-Gárgyán}. We derive the statement by  purely combinatorial, elementary arguments.  As a corollary, we deduce that no subdiagonal central sections of the unit cube are extremal, apart from the minimal, maximal, and the main diagonal sections. We also prove several consequences for Eulerian numbers. 
\end{abstract}

\maketitle

\section{Introduction}\label{sec:intro}

This article focuses on estimates of the \emph{Laplace--Pólya integral}, defined by 
\begin{equation}\label{eq:Jnr}
    J_n(r):=\frac{1}{\pi}\int_{-\infty}^\infty\sinc^nt\cdot\cos(r t)\dd t
\end{equation}
where $n\geq1$ is an integer, $r\in\R$, and $\sinc$ is the unnormalised cardinal sine function, i.e.
\begin{equation}\label{eq:sinc}
    \sinc x:=\begin{cases}\dfrac{\sin x}{x},&\text{if}\ x\neq0,\\1,&\text{if}\ x=0.\end{cases}
\end{equation}
As we illustrate in the next subsection, the function~\eqref{eq:Jnr} arises in various areas of mathematics, and through linking seemingly distant notions and quantities, it leads to interesting results in convex geometry as well as new, alternative proofs for results in enumerative combinatorics.

\subsection{A brief historical overview}

The very first appearance of the function~\eqref{eq:Jnr} is attributed to Laplace in the 19\textsuperscript{th} century. In his prominent work on probability theory~\cite{Laplace}, he pointed out that the probability density function of the sum of $n$ independent random variables which are uniformly distributed in $[-1,1]$ is given by half of the Laplace--Pólya integral~\eqref{eq:Jnr}.  Formally, if  $X_1,\ldots,X_n$ are i.i.d. uniform random variables on $[-1,1]$, and $f_X(.)$ denotes the probability density function of a continuous random variable $X$, then for $n\geq2$, or $n=1$ and $r\neq\pm1$,
\begin{equation}\label{eq:Jnr-density}
    f_{\sum_{i=1}^nX_i}(r)=\frac12 J_n(r)
\end{equation}
 --- nowadays, this is known as an Irwin--Hall distribution. Laplace~\cite[\textsection 42 in Chapter III]{Laplace} also proved that the probability density function above may also be expressed as a finite sum:
\begin{equation}\label{eq:Jnr-explicit}
    J_n(r)=\frac{1}{2^{n-1}(n-1)!}\sum_{i=0}^{\lfloor\frac{n+r}{2}\rfloor}(-1)^{i}\binom{n}{i}(n+r-2i)^{n-1},
\end{equation}
even though evaluating this formula becomes computationally difficult for large $n$'s. 
Since Laplace's time, justifying equation~\eqref{eq:Jnr-explicit} has become a standard exercise; see, e.g. \cite[Appendix III Example 22]{Bromwich} and \cite[Problem 1928]{Wolstenholme}. It can also be found in several integral tables, such as~\cite[Formula 1.6 (11)]{Bateman-Erdélyi} and~\cite[Formula 3.836.5]{inttábla-orosz}. 

Neither formula~\eqref{eq:Jnr} nor~\eqref{eq:Jnr-explicit} indicates that the Laplace--Pólya integral equals zero for $\abs{r}\geq n$, yet this fact readily follows from~\eqref{eq:Jnr-density}. Also note that since the probability density function of the sum of random variables is given by the convolution of their probability density functions, the closed support of~\eqref{eq:Jnr-density} is $[-n,n]$, and its values may be obtained by calculating the respective convolutions.

The story of the integral~\eqref{eq:Jnr} continues in the early 20\textsuperscript{th} century, with Pólya's work~\cite{Pólya} related to statistical mechanics. Let $Q_n=\big[-\frac12,\frac12\big]^n$ denote the centred $n$-dimensional cube, where $n \geq 2$. 
Hyperplanes through the origin $\zero{n}$  of $\R^n$ will be called \emph{central}; they are of the form  $\mbf u^\perp$ where $\mbf u\in S^{n-1}$ is a unit normal vector (as usual, $S^{n-1}$ denotes the unit sphere in~$\R^n$).  We introduce the \emph{central section function} as the volume of the section of the cube cut by $\mbf u^\perp$: 
\begin{equation}\label{eq:sigma}
    \sigma(\mbf u)=\Vol{n-1}{Q_n\cap \mbf u^\perp}
\end{equation}
for a unit vector $\mbf u \in S^{n-1}$.
Pólya~\cite{Pólya} showed that this quantity may be evaluated by the classical integral formula
\begin{equation}\label{eq:sigma-unit}
    \sigma(\mbf u)=\frac1\pi\int_{-\infty}^\infty\prod_{i=1}^n\sinc(u_it)\dd t
\end{equation}
where $\mbf u=(u_1,\ldots,u_n)$.  

A unit vector in $S^{n-1}$ is said to be a  \emph{$k$-diagonal direction} if it is parallel to the main diagonal of a $k$-dimensional face of $Q_n$. Let $\one{n}$ denote the $n$-dimensional vector $(1,\ldots,1)$. The standard $k$-diagonal direction is given by
\begin{equation}\label{eq:diagonal}
    \mbf d_{n,k}:=\frac{1}{\sqrt k}\big(\one{k},\zero{n-k}\big)
\end{equation}
for $k=1,\ldots,n$. Up to permuting the coordinates and flipping their signs, all $k$-diagonal directions are of the above form. Such normal vectors and the corresponding orthogonal central sections of $Q_n$ will simply be called \emph{diagonal}; $\mbf d_{n,n}$ is the \emph{main diagonal direction}, while $\mbf d_{n,k}$ is a \emph{subdiagonal} for $1 \leq k \leq n-1$. Based on \eqref{eq:sigma-unit}, the volume of central diagonal sections can be expressed by the central values of the Laplace--Pólya integral: for all $k=1,\ldots,n$,
\begin{equation}\label{eq:sigma-Jn0}
    \sigma(\mbf d_{n,k})=\sqrt kJ_k(0).
\end{equation}

Furthermore, $\sqrt{n} J_n(r)$ equals to the volume of the main diagonal section of $Q_n$ that lies at distance $\frac{r}{2 \sqrt{n}}$ from the center, see \cite[Formula (2.9)]{Ambrus-Gárgyán}. We also note that König and Koldobsky~\cite{KK} proved formulas for the surface area of cube sections that can be interpreted in terms of the Laplace--Pólya integrals.

Hadwiger~\cite{Hadwiger} and later Hensley~\cite{Hensley} showed that the central section function $\sigma(\mbf u)$ (cf. \eqref{eq:sigma}) is globally minimal over $S^{n-1}$ when $\mbf u=\mbf d_{n,1}$. In his seminal article, Ball~\cite{Ball} proved that $\mbf u=\mbf d_{n,2}$ constitutes a global maximum:
\[
\sqrt{n} J_n(0) \leq \sqrt{2}
\]
for every $n \geq 1$. As $n \to \infty$, we also have 
	\[\lim_{n\to\infty}\sqrt{n}J_n(0)=\sqrt{\frac{6}{\pi}},\]
 see  \cite{Laplace,Pólya}.  Furthermore, Bartha, Fodor and González Merino~\cite{BFG} proved that the convergence is strictly monotone increasing for $n\geq3$:
\begin{equation}\label{eq:BFG}
\sqrt{\frac n {n+1}} \leq \frac{J_{n+1}(0)}{J_n(0)}
\end{equation}
for every $n \geq 3$. This improves upon a bound of Lesieur and Nicolas~\cite{Lesieur-Nicolas} who, in their work regarding maxima of Eulerian numbers of the first kind, showed that for {\em even} positive values of $n$,
\begin{equation}\label{eq:LNLP}
\frac{n}{n+1}<\frac{J_{n+2}(0)}{J_n(0)}<\frac{n+1}{n+2}.
\end{equation}
Recall that the Eulerian number $A(m,l)$ for integers $m, l \geq 1$ is defined as the number of permutations of the set $\{1,\ldots,m\}$ in which exactly $l-1$ elements are greater than the previous element. The connection to the Laplace--Pólya integral is provided by the formula 
\[
A(m,l)=m!J_{m+1}(2l-m-1);
\]
for more details, see Section~\ref{sec:Eulerian}.

In the 20\textsuperscript{th} century, applications in telecommunication raised the need to calculate integrals in the vein of~\eqref{eq:Jnr} more efficiently. One way to evaluate $J_n(0)$ is  through its power series expansion. Determining asymptotic coefficients began with Silberstein's work~\cite{Silberstein-1940s} during WW2. Unfortunately, his paper contained several errors that had later been addressed in a series of articles~\cite{Burgess-1940s,Grimsey-1940s,Goddard-1940s,Parker-1940s}. Finally, the correct coefficients up to order $10$ were published by Medhurst and Roberts~\cite{Medhurst}. For central values (i.e. for $r=0$), the asymptotic expansion up to order $3$ reads as
\begin{equation*}
    J_n(0)=\sqrt{\frac{6}{\pi n}}\bigg(1-\frac{3}{20n}-\frac{13}{1120n^2}+\frac{27}{3200n^3}+O\Big(\frac{1}{n^4}\Big)\bigg).
\end{equation*}
Finer estimates were established in \cite{Kerman,Paris,Schlage}. The exact values of $J_n(0)$ were calculated by Harumi, Katsura and Wrench~\cite{Harumi-Katsura} up to $n=30$, and by Medhurst and Roberts~\cite{Medhurst} up to $n=100$ (see Table~\ref{tab:J2_10(0)} for the initial segment of the sequence.) Later, Butler~\cite{Butler} and Fettis~\cite{Fettis}\footnote{In \cite[Formula (7)]{Fettis} the factor $\frac2\pi$ in the second row should be $\frac2n$ and should appear on the right hand side of the equality sign.} presented further methods based on Poisson's summation formula to evaluate the Laplace--Pólya integral efficiently. 

The easiest way to calculate exact values of $J_n(r)$ is based on its recursive properties. Medhurst and Roberts~\cite{Medhurst} proved ``vertical'' recurrence relations for even and odd values of $r$, separately. Our work is based on the recursion proved by Thompson~\cite[p. 331, l. 2]{Thompson}:
\begin{equation}\label{eq:Jnr-recursion}
        J_n(r)=\frac{n+r}{2(n-1)}J_{n-1}(r+1)+\frac{n-r}{2(n-1)}J_{n-1}(r-1)
    \end{equation}
for $n\geq3$. Although this formula is valid for all $r\in\R$, we will use it only for integer values throughout this work. 
Since $J_n(r)$ is even in $r$, we specifically derive that
\begin{equation}\label{eq:Jn0}
    J_n(0)=\frac{n}{n-1}J_{n-1}(1)
\end{equation}
for $n\geq3$. Combined with \eqref{eq:Jnr-recursion} this provides a simple way to calculate the central values $J_n(0)$ for small $n$'s; see Table~\ref{tab:J2_10(0)}.

\begin{table}[h!] 
		\[\begin{NiceArray}{*{8}{c}}[hvlines-except-borders,cell-space-limits=5pt,columns-width =40 pt]
			n&2&3&4&5&6&7&8\\
			J_{n}(0)&1&\dfrac{3}{4}&\dfrac{2}{3}&\dfrac{115}{192}&\dfrac{11}{20}&\dfrac{5887}{11520}&\dfrac{151}{315}
		\end{NiceArray}\]
		\caption{Central values $J_n(0)$ for $n=2,\ldots,8$.}
		\label{tab:J2_10(0)}
	\end{table}

\subsection{New results}

The study of cube sections necessitates to analyse the decay of the sequence $\big(J_n(0)\big)_{n=1}^\infty$, see \cite{Ambrus-Gárgyán}. By the recurrence relations~\eqref{eq:Jnr-recursion} and \eqref{eq:Jn0}, this is equivalent to the study of the value of the Laplace--Pólya integral at integers.  In our previous work~\cite{Ambrus-Gárgyán}, we gave an upper bound on the decay of the two-step ratio of the Laplace--Pólya integral.

\begin{theorem}[{\cite[Theorem 1.4]{Ambrus-Gárgyán}}]\label{th:upper}
        Let $n\geq4$ and $r$ be integers satisfying $-1\leq r\leq n-2$. Then
        \begin{equation}\label{eq:Jnr-upper}
            \frac{J_n(r+2)}{J_n(r)}\leq d_{n,r}
        \end{equation}
        where
        \begin{equation}\label{eq:dnr}
            d_{n,r}=\frac{(n-r)}{(n+r+2)}\cdot\frac{(n-r-2)(n-r+2)}{(n+r)(n+r+4)}.
        \end{equation}
    \end{theorem}

In our present work, we complement the above result with a corresponding lower estimate.

\begin{theorem}\label{th:lower}
        Let $n\geq4$ and $r$ be integers satisfying $-1\leq r\leq n-2$. 
    Then
        \begin{equation}\label{eq:Jnr-lower}
            c_{n,r}\leq\frac{J_n(r+2)}{J_n(r)}
        \end{equation}
        where
        \begin{equation}\label{eq:cnr}
            c_{n,r}=\frac{(n-r)^2}{(n+r+2)^2}\cdot\frac{(4n-7r-8)(4n+3r+6)}{(4n+7r+6)(4n-3r)}.
        \end{equation}    
Moreover, equality holds in \eqref{eq:Jnr-lower} if and only if $r = -1$.
    \end{theorem}

Note that, although Theorem~\ref{th:lower} holds for the entire range $-1 \leq r \leq n-2$, the bound \eqref{eq:cnr} remains non-negative only for $-1\leq r\leq\big\lfloor\frac{4n-8}{7}\big\rfloor$. Consequently, the estimate becomes trivial for larger values of $r$.

As a direct consequence of Theorems~\ref{th:upper} and \ref{th:lower}, we obtain the following bounds.

\begin{corollary}\label{cor:bounds-central}
    For $n\geq4$ we have
    \begin{equation}\label{eq:J2/J0}
        \frac{n(n-2)}{(n+2)^2} < \frac{J_n(2)}{J_n(0)}\leq\frac{n-2}{n+4}.
    \end{equation}
    Based on \eqref{eq:Jnr-recursion} and \eqref{eq:Jn0}, this leads to the following estimate for central values:  \begin{equation}\label{eq:Jn+2/Jn}
        \frac{n}{n+1} < \frac{J_{n+2}(0)}{J_n(0)}\leq\frac{(n+2)(n^2+2n-2)}{n(n+1)(n+4)}
    \end{equation}
    where $n \geq 2$ for the lower bound, and $n\geq 4$ for the upper bound. 
\end{corollary}

The lower bound of \eqref{eq:Jn+2/Jn} for $n= 2,3$ follows from direct substitution based on Table~\ref{tab:J2_10(0)}. The upper bounds were established in \cite{Ambrus-Gárgyán}. Moreover, we note that \eqref{eq:BFG} leads to the inequality
\[
\sqrt{\frac{n}{n+2}}< \frac{J_{n+2}(0)}{J_n(0)}
\]
which is slightly stronger than the lower estimates in Corollary~\ref{cor:bounds-central}.

We readily see that the lower bound in \eqref{eq:Jn+2/Jn} extends the lower estimate of \eqref{eq:LNLP} to {\em all} values of $n \geq 2$, and also provides a simpler, purely combinatorial proof for even values of~$n$. We note that L. Pournin \cite{Pournin-private} managed to extend the upper estimate of \eqref{eq:LNLP} to all integers $n \geq 136$ by utilising the bound of Theorem 2.1 in \cite{Pournin-deep}. In fact, he showed that for those values of $n$, the stronger inequality
\[
\frac{J_{n+2}(0)}{J_n(0)} < \sqrt{\frac n {n+2}} \left(1 + \frac 1 {3 n^2} \right)
\]
holds as well. We believe that the extension of the  upper bound of \eqref{eq:LNLP} should also be proveable by combinatorial methods, which would amount to verifying the inequality 
 \begin{equation}\label{eq:J2/J0-better}
        \frac{J_n(2)}{J_n(0)}\leq\frac{n(n^2-2)}{(n+2)^3}
    \end{equation}
for all $n \geq 2$ --- for even values of $n$, this follows from the results in \cite{Lesieur-Nicolas}. The relationship between the quantities in \eqref{cor:bounds-central} and in \eqref{eq:J2/J0-better} is plotted in Figure~\ref{fig:central-ratio}.

\begin{figure}[h]
    \centering
    \includegraphics[scale=.6]{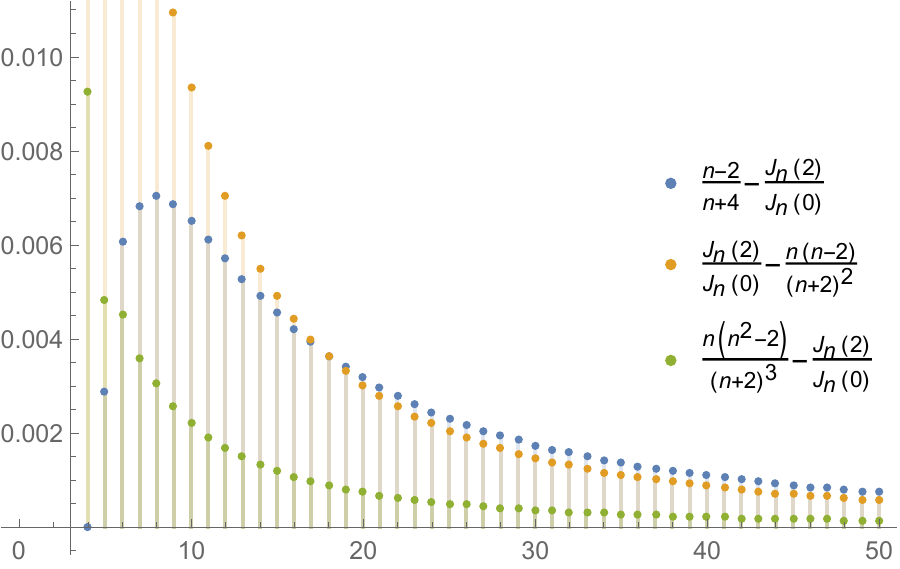}
    \caption{Difference between the central ratio and its estimates.}
    \label{fig:central-ratio}
\end{figure}

In Section~\ref{sec:cube} we prove an application of Theorem~\ref{th:lower} with respect to the volume of central sections of the cube $Q_n$. Theorem 1.1 in ~\cite{Pournin-local} states that the diagonal sections of $Q_n$ orthogonal to $\mbf d_{n,k}$, where $k=4,\ldots,n$, are strictly locally maximal with respect to the volume function, whenever $n\geq4$. In our previous work~\cite{Ambrus-Gárgyán}, we gave an alternative proof for this statement in the case $k = n$. i.e.  for the main diagonal sections, based on Theorem~\ref{th:upper}. In Section~\ref{sec:cube}, using the estimates from Corollary~\ref{cor:bounds-central}, we point out that Pournin's argument is unfortunately based on an incorrect  formula\footnote{The mistake in \cite{Pournin-local} is on p.570, where a summand of $2\lambda = -\sqrt{n} \frac{\partial}{\partial a_d^2} \frac{V}{\| a \|} $ is missing from the right side of the last equation on the page. This issue was addressed in the Corrigendum~\cite{pournin2025corrigendum}, where the correct results are proven.} by demonstrating that the subdiagonal sections are \emph{not} locally extremal for any $3 \leq k \leq n-1$:
\begin{theorem}\label{th:subdiagonal}
    For any $n\geq4$, the subdiagonal sections of $Q_n$ of order $3 \leq k \leq n-1$ do not have locally extremal volumes among the central sections of $Q_n$.
\end{theorem}

Finally, in Section~\ref{sec:Eulerian}, we elaborate on the connection between Eulerian numbers and the Laplace--Pólya integral. Continuing the work of Lesieur and Nicolas \cite{Lesieur-Nicolas}, we establish estimates for the ratios of consecutive Eulerian numbers (Corollary~\ref{cor:Eulerian-ratio}) and prove convexity properties of those (Propositions~\ref{prop:f-odd} and \ref{prop:f-even}).

\section{Lower bound on the two-step ratio}\label{sec:ratio}

\begin{proof}[Proof of Theorem~\ref{th:lower}]

    Suppose that $r\geq-1$ is an integer. We will proceed by induction on $n$, with $n=4$ being the base case.

    The values of $J_4(r)$ for $r = -1,\ldots, 4$ can be calculated directly using the formula~\eqref{eq:Jnr-explicit}. By substituting these, the $n = 4$ case of \eqref{eq:Jnr-lower} is easy to check, see Table~\ref{tab:lower-ineq-J4k}.
    
    \begin{table}[h]
        \[\begin{NiceArray}{*{5}{c}}[hvlines-except-borders,cell-space-limits=5pt, columns-width =50 pt]
	    r&-1&\ 0 & 1 & 2 \\
            \frac{J_4(r+2)}{J_4(r)}&1&\frac14&\frac1{23}&0\\
            c_{4,r}&1&\frac29&\frac{225}{18473}&-\frac{7}{240}
	    \end{NiceArray}\]
	    \caption{Two sides of the inequality \eqref{eq:Jnr-lower} for $n=4$ and $r=-1,\ldots,2$.}
	    \label{tab:lower-ineq-J4k}
	    \end{table}

For later reference, note that \eqref{eq:cnr} shows that in the range $-1\leq r\leq n-2$, the bound $c_{n,r}$ is positive iff $r < \frac{4 n -8}{7}$.
    
    Suppose now that for some $n\geq4$, \eqref{eq:Jnr-lower} holds for each $-1\leq r\leq n-2$. We need to show that
    \begin{equation}\label{eq:Jn+1-induction}
        c_{n+1,r}J_{n+1}(r)\leq J_{n+1}(r+2)
    \end{equation}
    for each $r$ with $-1\leq r\leq n-1$, with strict inequality for $0 \leq r $.

    For $r=-1$, equality holds above, because $c_{n,-1}=\frac{J_n(1)}{J_n(-1)}=1$ for each $n\geq4$. Also, if $r\geq \frac{4n-4}{7}$, then $c_{n+1,r}\leq0$, hence \eqref{eq:Jnr-lower} is trivially satisfied.

    Assume that $0\leq r < \frac{4n-4}{7}$. Then $c_{n,r-1}>0$, as noted before. 
    By \eqref{eq:Jnr-recursion} we have
    \begin{align*}
        \begin{split}
            J_{n+1}(r+2)=\frac{n+r+3}{2n}J_n(r+3)+\frac{n-r-1}{2n}J_n(r+1),\\
            J_{n+1}(r)=\frac{n+r+1}{2n}J_n(r+1)+\frac{n-r+1}{2n}J_n(r-1).
        \end{split}
    \end{align*}
    Applying the induction hypothesis for the pairs $(n,r+1)$ and $(n,r-1)$ shows that
    \begin{align*}
        \begin{split}
            c_{n,r+1}J_n(r+1)\leq J_n(r+3),\  \ \text{ and }\\
            J_n(r-1)\leq\frac{1}{c_{n,r-1}} J_n(r+1).
        \end{split}
    \end{align*}
    Thus, \eqref{eq:Jn+1-induction} follows from the inequality
    \[c_{n+1,r}\Big(\frac{n+r+1}{2n}+\frac{n-r+1}{2n}\cdot\frac{1}{c_{n,r-1}}\Big)<\frac{n+r+3}{2n}c_{n,r+1}+\frac{n-r-1}{2n},\]
    which is equivalent to 
    \begin{equation}\label{eq:cnr-ineq}
        0<(n+r+3)c_{n,r+1}+(n-r-1)-c_{n+1,r}\Big((n+r+1)+(n-r+1)\frac{1}{c_{n,r-1}}\Big).
    \end{equation}
   Substituting \eqref{eq:cnr}, the right-hand side can be combined to a single fraction with the numerator
    \begin{multline}\label{eq:lower-num}
        8(r+1)\Big(128n^6+128n^5(41r^2+82r+36)+8n^4(1807r^2+3614r+1585)-\\[5pt]-4n^3(1018r^4+4072r^3+5441r^2+2738r+45)-4n^2(958r^4+3832r^3+9243r^2+10822r+3697)+\\[5pt]+3n(r+1)^2(455r^4+1820r^3+2299r^2+958r-1476)-9(r+1)^2(91r^4+364r^3-17r^2-762r-220)\Big),
    \end{multline}
  and denominator
    \begin{equation}\label{eq:lower-den}
        (4n-7r-1)(4n-3r-3)(4n-3r+4)(n+r+3)^2(4n+3r+3)(4n+7r+10)(4n+7r+13).
    \end{equation}
    Clearly, \eqref{eq:lower-den} is positive as $r< \frac{4n -1}{7}$ and thus all factors are strictly positive. 
    
    Thus, it suffices to prove the positivity of~\eqref{eq:lower-num}. To that end, introduce the bivariate polynomial 
    \[p_n(s)=\alpha_ns^3+\beta_ns^2+\gamma_ns+\delta_n,\]
    whose coefficients are
    \begin{align*}
        \alpha_n&=1365n-819,\\
        \beta_n&=-4072n^3-3832n^2-1293n+5067,\\
        \gamma_n&=5248n^5+14456n^4+2668n^3-13980n^2-4500n-2268,\\
        \delta_n&=128n^6-640n^5-1776n^4+1224n^3+3024n^2.
    \end{align*}
   By direct calculation, the quantity \eqref{eq:lower-num} simplifies to
    \[8(r+1)p_n\big((r+1)^2\big).\]
   We will show that $p_n(s)>0$ for any $1\leq s $. To that end, note that 
    \begin{align}\label{eq:pn1}
    p_n(1)&=128n^6+4608n^5+12680n^4-180n^3-14788n^2-4428n+1980 = \notag \\
     &=4 (n-1) (n+1) (n+33 ) (2n+5 ) (4 n-1) ( 4 n+3) >\\
     &> 0 \notag
    \end{align}
    for all $n\geq4$. Furthermore, 
    \[p_n'(s)=3\alpha_ns^2+2\beta_n s+\gamma_n\]
    whose leading coefficient $3\alpha_n$ is strictly positive for all $n\geq1$. The discriminant of the above quadratic polynomial is given by
    \begin{align*}
        4\beta_n^2-12\alpha_n\gamma_n&=-19637504n^6-60380704n^5+199229392n^4+129789120n^3 - \\
        &\quad \quad-21331996n^2-59489208n+80408052<\\
        &< 10^8 (-n^6  - 6 n^5 + 21 n^4 + 18 n^3)= \\
        &=   10^8(3 -n) n^3 (n^2+ 9 n + 6)<0
    \end{align*}
    for $n\geq4$. Thus, $p_n'(s)>0$ for all $n\geq1$ and $s \geq 1$. Combining with \eqref{eq:pn1}, we derive that $p_n(s)>0 $ for all $n\geq4$ and $s \geq 1$, which completes the proof.
    \qedhere
        
\end{proof}

\section{Local maximality of diagonal sections}\label{sec:cube}

Before turning to the proof of Theorem~\ref{th:subdiagonal}, we briefly summarise a few technical details. Recall that we may evaluate the central section function $\sigma(\mbf u)$ (cf. \eqref{eq:sigma}) for {\em unit vectors} $\mbf u \in S^{n-1}$ via the integral~\eqref{eq:sigma-unit}. Note that this integral formula allows us to extend the support of $\sigma(.)$ to the whole $\R^n$ --- yet, for $\mbf v \not \in S^{n-1}$, the quantity $\sigma(\mbf v)$ does not express the volume of the central section of $Q_n$ anymore: for $\mbf v \in \R^n\setminus \{ \mbf 0_n\}$,
\begin{equation*}\label{eq:sigmav}
    \sigma({\mbf v})=\frac{1}{\pi}\int_{-\infty}^\infty\prod_{i=1}^n\sinc(v_i t)\dd t = \frac{1}{\abs{\mbf v}} \Vol{n-1}{Q_n\cap \mbf v^\perp},
\end{equation*}
see \cite[Section 2]{Ambrus-Gárgyán}.

It can be shown that the function $\sigma(\mbf v)$ defined above is differentiable at points $\mbf v \in \R^n$ with at least three non-zero coordinates
(see Pournin's works~\cite{Pournin-shallow}, \cite[Corollary 2.3]{Pournin-local}), and  for such $\mbf v=(v_1,\ldots,v_n)\in  \R^n$,
\begin{equation}\label{eq:pd}
    \frac{\partial}{\partial v_j}\sigma(\mbf v)=\frac1\pi\int_{-\infty}^\infty\prod_{i\neq j}\sinc(v_i t)\cdot\frac{\cos(v_j t)-\sinc(v_j t)}{v_j}\dd t
\end{equation}
for any $j=1,\ldots,n$  (cf. \cite[formula (2.11)]{Ambrus-Gárgyán}). Furthermore, if $\mbf v$ has  at least four non-zero coordinates, then $\sigma(\mbf v)$ is twice differentiable by the Newton-Leibniz rule applied to \eqref{eq:pd},  cf. \cite[Corollary 2.4]{Pournin-local}.

Critical points of $\sigma(.)$ on $S^{n-1}$ are called {\em critical directions} --- note that these correspond to central sections which constitute critical points of the volume function. Based on the Lagrange multiplier method, the following characterisation was given in~\cite{Ambrus} (for an alternative proof, see \cite{Ivanov}):
\begin{proposition}[{\cite[Formula (2.12)]{Ambrus}}]\label{prop:Equ-2.12}
    The unit vector $\mbf u=(u_1,\ldots, u_n)\in S^{n-1}$ with at least three non-zero coordinates is a critical direction with respect to the central section function $\sigma(\mbf u)$ if and only if 
    \begin{equation}\label{eq:crit-char}
        \sigma(\mbf u)=\frac{1}{\pi(1-u_j^2)}\int_{-\infty}^\infty\prod_{i\neq j}\sinc(u_it)\cdot\cos(u_jt)\dd t
    \end{equation}
    holds for each $j=1,\ldots,n$.
\end{proposition}
The argument also yields that at such critical directions $\mbf u\in S^{n-1}$,
\begin{equation}\label{eq:sigmapartial}
    \frac{\partial}{\partial u_j}\sigma(\mbf u)=-\sigma(\mbf u)\cdot u_j,
\end{equation}
see \cite[Proof of Theorem 1]{Ambrus}. Accordingly, the Lagrange function
\begin{equation}\label{eq:lagrangefunction}
    \Lambda(\mbf v)=\sigma(\mbf v)+\tilde\lambda\cdot\big(\abs{\mbf v}^2-1\big)
\end{equation}
defined on $\R^n$ has a stationary point at $\mbf v=\mbf u\in S^{n-1}$ with the Lagrange multiplier
\begin{equation}\label{eq:lagrangempl}
    \tilde\lambda=\frac{\sigma(\mbf u)}{2}.
\end{equation}

\begin{proof}[Proof of Theorem~\ref{th:subdiagonal}]

First, we assume that $k \geq 4$  --- this will ensure the required differentiability properties.
In the last part of the proof, we will handle the $k=3$ case separately.

Consider now  the subdiagonal directions $\mbf{d}_{n,k} = \frac 1{\sqrt{k}} (\one{k},\zero{n-k})$ for arbitrary $k=4,\ldots,n-1$.
Since Proposition~\ref{prop:Equ-2.12} and \eqref{eq:Jn0} imply that $\mbf d_{n,k}$  is a critical direction, we may apply the formulae \eqref{eq:sigmapartial},  \eqref{eq:lagrangefunction} and \eqref{eq:lagrangempl}. We will show that the function $\sigma(\mbf v)$ has no local extremum at $\mbf v = \mbf d_{n,k}$ subject to the constraint $\abs{\mbf v} =1$.  
This is a constrained local optimisation problem which can be solved by studying the bordered Hessian matrix, i.e. the Hessian of the Lagrange function $\Lambda(\mbf v)$ defined by \eqref{eq:lagrangefunction}. 

Based on \eqref{eq:lagrangefunction} and \eqref{eq:lagrangempl},  one can calculate the bordered Hessian matrix $H(\Lambda(\mbf v))$ for general $\mbf v \in \R^n$ where $\sigma(.)$ is twice differentiable. Using \eqref{eq:pd} and the subsequent remark,
the Newton-Leibniz rule implies that for such $\mbf v$, $H(\Lambda(\mbf v))$ takes the following form: 
\begin{equation}\label{eq:borderedHessian}
H(\Lambda(\mbf v))=
\begin{bNiceArray}{*{5}{c}}[margin,cell-space-limits=0pt] 
            0&2v_1&2v_2& \hspace{1 cm}\ldots &\hspace{1 cm} 2v_n\\
            2v_1&\Block{4-4}{\dfrac{\partial^2\sigma}{\partial v_{j_1} \partial v_{j_2}}(\mbf v)+\sigma(\mbf v)\cdot\begin{cases}
                0,&\text{if}\ j_1\neq j_2 \\
                1,&\text{if}\ j_1=j_2                 
            \end{cases}}&&&\\
            2v_2&&&&\\
            \vdots&&&&\\
            2v_n&&&&
        \end{bNiceArray}  =
        \begin{bNiceMatrix}
            0&2v_1&2v_2&\ldots&2v_n\\
            2v_1& \beta_1(\vv) & \gamma_{1,2}(\mbf v) & \ldots& \gamma_{1,n}(\mbf v) \\
            2v_2& \gamma_{2,1}(\mbf v) & \beta_2(\vv) &  \ldots& \gamma_{2,n}(\mbf v) \\
             \vdots&\vdots&\vdots&\ddots&\vdots\\
            2v_n & \gamma_{n,1}(\mbf v) & \gamma_{n,2}(\mbf v)&  \ldots& \beta_n(\mbf v)
        \end{bNiceMatrix},
\end{equation}
where
\begin{equation}\label{eq:betavdef}
   \beta_j(\vv) =
    \frac 1 \pi \int_{-\infty}^\infty\prod_{i\neq j}\sinc (v_i t)\cdot\bigg(\frac{2}{v_j^2}\Big(\sinc (v_jt)-\cos (v_jt)\Big)-\frac{(v_jt)^2}{v_j^2}\sinc (v_jt)+\sinc (v_jt)\bigg)\dd t
\end{equation}   
      along the diagonal $j=j_1=j_2$ with $1 \leq j \leq n$, and
\begin{equation}\label{eq:gammavdef}
    \gamma_{j_1,j_2}(\mbf v)   = \frac 1 \pi \int_{-\infty}^\infty\prod_{i\neq j_1,j_2}\sinc (v_it)\cdot\frac{\cos (v_{j_1}t)-\sinc (v_{j_1}t)}{v_{j_1}}\cdot\frac{\cos (v_{j_2}t)-\sinc (v_{j_2}t)}{v_{j_2}}\dd t
\end{equation}
        for the off-diagonal entries, i.e. $j_1\neq j_2$, $1 \leq j_1, j_2 \leq n$.
        
Substitute now $\mbf v=\mbf{d}_{n,k}$. Let $H_{m,k}$ denote the principal upper left minor of $H\big(\Lambda(\mbf{d}_{n,k})\big)$ of order $m$ for\footnote{The upper bound of $m$ in~\cite[Section 4]{Ambrus-Gárgyán} should read as $n+1$, like in the present argument, instead of $n$.} $m=3,\ldots,n+1$.  We will apply the second derivative test for constrained local extrema~\cite[Theorem 1]{Spring}. This implies that if the sequence 
$H_{3, k}, \ldots, H_{n+1, k}$ contains two non-zero consecutive elements of the same sign as well as two non-zero consecutive elements with alternating sign, then $\sigma(.)$ does not have a local extremum at~$\mbf{d}_{n,k}$ on the constraint set $S^{n-1}$. We will prove that this property holds for each $k=4,\ldots,n-1$, which shows that these cases do not represent local extrema.  (Besides differentiability issues, an additional reason for treating the $k=3$ case separately is that the second derivative test is inconclusive then.) 

Integrating by substitution for $v_j t$ in \eqref{eq:betavdef} and \eqref{eq:gammavdef}, applying the identity $\sin^2t=\frac{1-\cos(2t)}{2}$ and  employing \eqref{eq:Jnr-recursion} repeatedly  results in the formulae

  \begin{align*}
            \beta_j(\mbf{d}_{n,k})
            &=
            \frac{\sqrt k}{\pi} \cdot \int_{-\infty}^\infty\Big(\sinc^{k-1}t\Big) \cdot\Big(2k(\sinc t-\cos t)-k t^2\sinc t+\sinc t\Big)\dd t=\\[5pt]&=
            {\sqrt k} \cdot \bigg((2k+1)J_k(0)-2k J_{k-1}(1)-\frac k2 J_{k-2}(0)+\frac k2 J_{k-2}(2)\bigg)=\\[5pt]&=
            {\sqrt k} \cdot \bigg(\Big(\frac{(2k+1)k}{2(k-1)}-k-\frac k2\Big)\cdot J_{k-2}(0) + \Big(\frac{(2k+1)k^2}{2(k-2)(k-1)}-\frac{k^2}{(k-2)}+\frac k2\Big)\cdot J_{k-2}(2)\bigg)=\\[5pt]&=
            \frac{k^{\frac32}}{2 (k-1)} \cdot \Big((4-k)J_{k-2}(0)+\frac{k^2+2}{k-2}J_{k-2}(2)\Big)
        \end{align*}        
for any $j=1, \ldots, k$, and 
        \begin{align*}
           \gamma_{{j_1},j_2}(\mbf{d}_{n,k})
            &=
            \frac{k^{\frac32}}{\pi}\cdot\int_{-\infty}^\infty\Big(\sinc^{k-2}t\Big) \cdot \Big(\cos^2t-2\sinc t\cdot\cos t+\sinc^2t\Big)\dd t=\\[5pt]&=
             {k^{\frac32}}\cdot\Big(J_k(0)-2J_{k-1}(1)+\frac12J_{k-2}(0)+\frac12J_{k-2}(2)\Big)=\\[5pt]&=
            {k^{\frac32}}\cdot\bigg(\Big(\frac{k^2}{2(k-2)(k-1)}-\frac{k}{k-2}+\frac12\Big)J_{k-2}(2)+\Big(\frac{k}{2(k-1)}-1+\frac12\Big)\cdot J_{k-2}(0)\bigg)=\\[5pt]&=
             \frac{k^{\frac32}}{ 2(k-1)} \cdot  \big(J_{k-2}(0)-J_{k-2}(2)\big)
        \end{align*}
for $j_1 \neq j_2$, $1 \leq j_1,j_2 \leq k$. Moreover, based on the limits
\[\lim_{v_j\to0}\frac{\cos (v_jt)-\sinc (v_jt)}{v_j}=0\quad\text{and}\quad\lim_{v_j\to0}\frac{2}{v_j^2}\Big(\sinc(v_jt)-\cos(v_jt)\Big)=\frac{2t^2}{3},\]
for any $j=k+1,\ldots,n$ we have 
\begin{align*}
    \beta_j(\mbf d_{n,k})&=
    \frac{\sqrt k}{\pi}\int_{-\infty}^\infty\sinc^kt\cdot\Big(\frac{2kt^2}{3}-kt^2+1\Big)\dd t=\\[5pt]&=
    \sqrt k\Big(J_k(0)-\frac k6J_{k-2}(0)+\frac k6J_{k-2}(2)\Big)=\\[5pt]&=
    \sqrt k\bigg(\Big(\frac{k^2}{2(k-2)(k-1)}-\frac k6\Big)J_{k-2}(0)+\Big(\frac{k}{2(k-1)}+\frac k6\Big)J_{k-2}(2)\bigg)=\\[5pt]&=
    \frac{k^{\frac32}}{6(k-1)}\Big(({4-k})J_{k-2}(0)+\frac{k^2+2}{k-2}J_{k-2}(2)\bigg)
\end{align*}
and $\gamma_{j_1,j_2}(\mbf d_{n,k})=0$ for $j_1\neq j_2$ and $j_1\geq k+1$ or $j_2\geq k+1$. In conclusion, \eqref{eq:borderedHessian} shows that $H\big(\Lambda(\mbf{d}_{n,k})\big)$ is of the form
\begin{equation}\label{eq:Hessian-a}
H\big(\Lambda(\mbf{d}_{n,k})\big)=
\begin{bNiceMatrix}[first-row,last-col]
&  \Hdotsfor{4}^{k} & \Hdotsfor{4}^{n-k}\\            
0&\alpha_k&\alpha_k&\cdots&\alpha_k&\alpha_k&\alpha_k&\cdots&\alpha_k\\       
\alpha_k&\beta_k&\gamma_k&\cdots&\gamma_k&0&0&\cdots&0&\Vdotsfor{4}^{k}\\     
\alpha_k&\gamma_k&\beta_k&\cdots&\gamma_k&0&0&\cdots&0\\           
\vdots&\vdots&\vdots&\ddots&\vdots&\vdots&\vdots&\ddots&\vdots\\
\alpha_k&\gamma_k&\gamma_k&\cdots&\beta_k&0&0&\cdots&0\\
\alpha_k&0&0&\cdots&0&\frac{\beta_k}3&0&\cdots&0&\Vdotsfor{4}^{n-k}\\
\alpha_k&0&0&\cdots&0&0&\frac{\beta_k}3&\cdots&0\\
\vdots&\vdots&\vdots&\ddots&\vdots&\vdots&\vdots&\ddots&\vdots\\
\alpha_k&0&0&\cdots&0&0&0&\cdots&\frac{\beta_k}{3}
\end{bNiceMatrix}_{\raisebox{-1.5mm}{$\scriptstyle (n+1)\times (n+1)$} }
\end{equation}
where\footnote{There is a typo in \cite[Formula (4.4)]{Ambrus-Gárgyán}, in the definitions of the quantities similar to the ones defined below. The power of the first factor in $\beta$ and $\gamma$ should be $1$.}
        \begin{equation}\label{eq:abg}
        \begin{gathered}
            \alpha_k=\frac{2}{\sqrt{k}},\quad \beta_k=  \frac{k^{\frac32}}{ 2(k-1)}  \cdot\Big((4-k)J_{k-2}(0)+\frac{k^2+2}{k-2}J_{k-2}(2)\Big),\\
            \text{and}\quad\gamma_k=\frac{k^{\frac32}}{ 2(k-1)}  \cdot\Big(J_{k-2}(0)-J_{k-2}(2)\Big)
        \end{gathered}
        \end{equation}
are functions of $k$.

Note that $ \beta_4 = 0$. Furthermore, the inequalities $\beta_k >0$  for $ k \geq 5$ and $\gamma_k >0$  for $k \geq 4$ are implied by direct calculation based on Table~\ref{tab:J2_10(0)} for $k =4,5$ and by Corollary~\ref{cor:bounds-central} for $k \geq 6$.

   Recall that $H_{m,k}$ denotes the $m$th principal minor of $ H\big(\Lambda(\mbf{d}_{n,k})\big)$. For $m=3,\ldots, k+1$, $H_{m,k}$ has the form
        \begin{equation}\label{eq:H-det}
            H_{m,k}=\begin{vNiceMatrix}
            0&\alpha_k&\alpha_k&\cdots&\alpha_k\\
            \alpha_k&\beta_k&\gamma_k&\cdots&\gamma_k\\
            \alpha_k&\gamma_k&\beta_k&\cdots&\gamma_k\\
            \vdots&\vdots&\vdots&\ddots&\vdots\\
            \alpha_k&\gamma_k&\gamma_k&\cdots&\beta_k
            \end{vNiceMatrix}_{m\times m}.
        \end{equation}
        Subtracting the second row from the ones below, expanding the resulting determinant along the first column, and finally adding the sum of all but the first columns of the remaining matrix to the first one results in an upper triangular matrix: 
 \begin{equation}\label{eq:Hmk}
 H_{m,k}=(-\alpha_k)\cdot\begin{vNiceMatrix}
            (m-1)\alpha_k&\alpha_k&\alpha_k&\cdots&\alpha_k\\
            0&\beta_k-\gamma_k&0&\cdots&0\\
            0&0&\beta_k-\gamma_k&\cdots&0\\
            \vdots&\vdots&\vdots&\ddots&\vdots\\
            0&0&0&\cdots&\beta_k-\gamma_k
        \end{vNiceMatrix}_{(m-1)\times (m-1)}\hspace{-1.25cm}= -(m-1)\alpha_k^2 \delta_k^{m-2}
        \end{equation}
    where  $\delta_k=\beta_k-\gamma_k$.
    
    Based on~\eqref{eq:abg}  we  derive that $\delta_4 = -\frac 4 3$ and $\delta_5 = -\frac{5 \sqrt{5}}{32}.$   
    For $k \geq 6$, note that according to Corollary~\ref{cor:bounds-central},
    \[
    J_{k-2}(2) \leq \frac{k-4}{k+2} J_{k-2}(0),
    \]
       hence, by \eqref{eq:abg},
        \begin{equation*}
        \delta_k=\beta_k-\gamma_k=  
        \frac{k^{\frac32}}{ 2(k-1)}\cdot\Big((3 - k)J_{k-2}(0)+\frac{k(k+1)}{k-2}J_{k-2}(2)\Big)\leq
        \frac{k^{\frac32}}{ 2(k-1)}\cdot\Big(- \frac{12}{k^2 - 4} J_{k-2}(0)\Big) < 0.
        \end{equation*}
        Therefore, $\delta_k<0$ for each $k \geq 4$, thus \eqref{eq:Hmk} implies that
\begin{equation}\label{eq:Hmk_sign}
(-1)^{m-1}H_{m,k}>0 \textrm{ for every  }m=3,\ldots,k+1.
\end{equation}

       Consider now the remaining minors of $H\big(\Lambda(\mbf d_{n,k})\big)$, that is, $H_{m,k}$ with $k+2 \leq m \leq n+1$. According to \eqref{eq:Hessian-a}, these are of the form 
                 \begin{equation*}\label{eq:Hessian-m}
            H_{m,k}=
            \begin{bNiceMatrix}[first-row,last-col]
            &\Hdotsfor{4}^{k} & \Hdotsfor{4}^{m-k-1}\\
            0&\alpha_k&\alpha_k&\cdots&\alpha_k&\alpha_k&\alpha_k&\cdots&\alpha_k\\
            \alpha_k&\beta_k&\gamma_k&\cdots&\gamma_k&0&0&\cdots&0&\Vdotsfor{4}^{k}\\
            \alpha_k&\gamma_k&\beta_k&\cdots&\gamma_k&0&0&\cdots&0\\
            \vdots&\vdots&\vdots&\ddots&\vdots&\vdots&\vdots&\ddots&\vdots\\
            \alpha_k&\gamma_k&\gamma_k&\cdots&\beta_k&0&0&\cdots&0\\
            \alpha_k&0&0&\cdots&0&\frac{\beta_k}3&0&\cdots&0&\Vdotsfor{4}^{m-k-1}\\
            \alpha_k&0&0&\cdots&0&0&\frac{\beta_k}3&\cdots&0\\
            \vdots&\vdots&\vdots&\ddots&\vdots&\vdots&\vdots&\ddots&\vdots\\
            \alpha_k&0&0&\cdots&0&0&0&\cdots&\frac{\beta_k}3
            \end{bNiceMatrix}_{\raisebox{-1.5mm}{$\scriptstyle m\times m$} }
        \end{equation*}

For $k=4$, we have $\beta_4 =0$, so $H_{m,4}$ can be easily calculated: $H_{4,4} = -3$, $H_{5,4}=4$ and $H_{6,4}=3$. This shows that $\mbf d_{n,4}$ is not locally extremal.

     From now on, assume that $k \geq 5$. Then, as we noted above, $\beta_k>0$, $\gamma_k>0$, and $\delta_k <0$.  Note that the last row of the determinant above contains only two non-zero elements: $\alpha_k$ and $\frac {\beta_k} 3$.      
       Accordingly, expanding the determinant along the last row yields that for  $k+2 \leq m \leq n+1$,
        \begin{equation}\label{eq:Hmk_2}
H_{m,k} = (-1)^{m+1} \alpha_k \cdot  (-1)^{m} \alpha_k \cdot \left ( \frac {\beta_k} 3\right)^{m-k-2} \cdot \begin{vNiceMatrix}
            \beta_k&\gamma_k&\cdots&\gamma_k\\
            \gamma_k&\beta_k&\cdots&\gamma_k\\
            \vdots&\vdots&\ddots&\vdots\\
            \gamma_k&\gamma_k&\cdots&\beta_k
            \end{vNiceMatrix}_{k\times k} + \frac {\beta_k} 3 H_{m-1,k} .
        \end{equation}
The above determinant can be calculated by subtracting the first row from the other ones and then adding the sum of the last $k-1$ columns to the first column, which results in an upper triangular determinant. Consequently,
\begin{equation*}
\begin{vNiceMatrix}
            \beta_k&\gamma_k&\cdots&\gamma_k\\
            \gamma_k&\beta_k&\cdots&\gamma_k\\
            \vdots&\vdots&\ddots&\vdots\\
            \gamma_k&\gamma_k&\cdots&\beta_k
            \end{vNiceMatrix}_{k\times k} =
            \begin{vNiceMatrix}
            \beta_k+(k-1)\gamma_k&\gamma_k&\cdots&\gamma_k\\
            0&\beta_k-\gamma_k&\cdots&0\\
            \vdots&\vdots&\ddots&\vdots\\
            0&0&\cdots&\beta_k-\gamma_k
            \end{vNiceMatrix}_{k\times k}
            =
            (\beta_k + (k-1)\gamma_k) \delta_k^{k-1}
\end{equation*}
whose sign is $(-1)^{k-1}$. Therefore, the first summand on the right hand side of \eqref{eq:Hmk_2} has sign $(-1)^k$. Since \eqref{eq:Hmk_sign} implies that $(-1)^k H_{k+1,k}>0$,  induction on $m$ shows that the sign of $H_{m,k}$ is $(-1)^k$ for every $m=k+1,\ldots,n+1$.
Along with \eqref{eq:Hmk_sign}, this shows that the sequence of minors is neither alternating nor of constant sign, hence $\mbf d_{n,k}$ is not locally extremal.

Finally, we handle the case $k=3$. Clearly, it suffices to prove that $\mbf d_{3,3}$ is not a local extremal direction in $S^2$ --- equivalently, the main diagonal sections of $Q_3$ do not have locally extremal volume. We will show this by a direct, elementary geometric argument (we also note that it can be proved by the methods in \cite{KK-3dim}, see the Remark after Proposition 7 therein).

\begin{figure}[h]
    \centering
    \includegraphics[width = .7\textwidth]{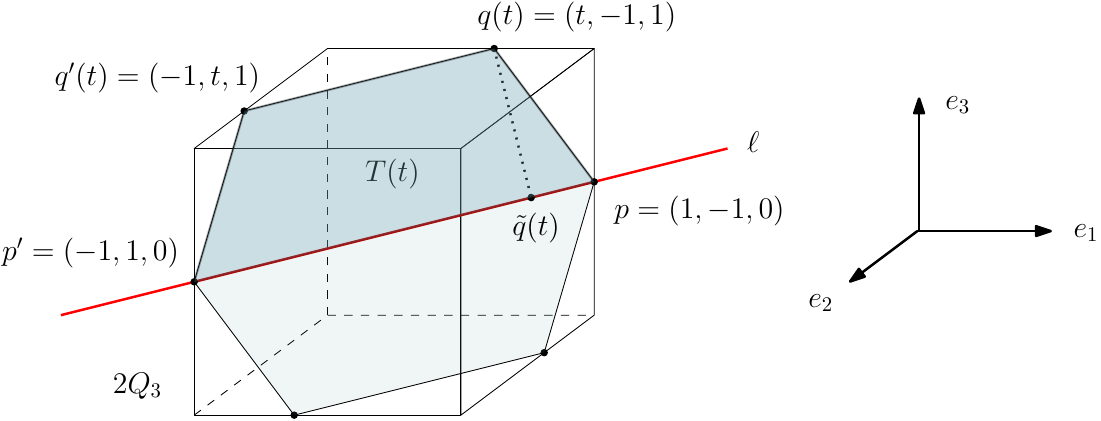}
    \caption{The main diagonal section of $Q_3$ is not locally extremal.}
    \label{fig:cube}
\end{figure}

To simplify the notation, we verify the statement for the cube $2 Q_3 = \big[-1,1\big]^3$. Let $H=\mbf d_{3,3}^\perp$.  Then the main diagonal section $2 Q_3 \cap H$ is a regular hexagon with two opposite vertices at $p=(1, -1, 0)$ and $p'=(-1, 1, 0)$, see Figure~\ref{fig:cube}. Let $\ell$ be the line through $p$ and $p'$, and rotate $H$ around the axis $\ell$. This may be parameterised as follows: for $t \in [-1, 1 ]$, let $q(t)=(t, -1, 1)$, $q'(t)=(-1,t, 1)$ and denote by $H(t)$ the plane spanned by the points $p, p'$ and $q(t)$ --- note that $H = H(0)$. Then $2 Q_3 \cap H(t)$ is a hexagon that decomposes to the union of two congruent trapezia, one of which is $T(t) = p p' q'(t) q(t)$. To calculate the area of $T(t)$, let $\tilde q(t)$ denote the orthogonal projection of $q(t)$ onto $\ell$. Note that 
$|p p'|= 2 \sqrt{2}$  and $| q(t) q'(t)| = \sqrt{2}(1 + t)$, 
hence by symmetry, $|\tilde q(t) p| = \frac{1 -t}{\sqrt{2}}$. Since $|p q(t) | = \sqrt{1 + (1-t)^2}$, the Pythagorean theorem implies that 
\[
|q(t) \tilde q(t)| = \sqrt{1 + \frac{(1-t)^2}{2}}
\]
and thus the area of $T(t)$ is 
\[
a(t):= \Vol{2}{T(t)} = \frac{3+t}{\sqrt{2}} \cdot  \sqrt{1 + \frac{(1-t)^2}{2}}.
\]
Direct calculation shows that $a'(0) = a''(0) = 0$ while $a^{(3)}(0) = \frac 2 {\sqrt{3}} >0$. Consequently, the function $a(t)$, and thus the area of $2 Q_3 \cap H(t)$, cannot have a local extremum at $t=0$.
\end{proof}

\section{Consequences for Eulerian numbers}\label{sec:Eulerian}

A combinatorial interpretation of the Laplace--Pólya integral originates from the connection with Eulerian numbers of the first kind $A(m,l)$, which will simply be referred to as Eulerian numbers (see also \cite[Section 3]{Ambrus-Gárgyán}). The classic definition of $A(m,l)$ is the number of permutations of the set $\{1,\ldots,m\}$ in which exactly $l-1$ elements are greater than the previous element. 
This immediately implies the symmetry relation 
 \begin{equation}\label{eq:Euler-symm}
	    A(m,l)=A(m,m-l+1).
    \end{equation}  
Eulerian numbers can be explicitly expressed by the formula
\begin{equation}\label{eq:Euler-explicit}
    A(m,l)=\sum_{i=0}^l(-1)^{i}\binom{m+1}{i}(l-i)^m,
\end{equation}
see ~\cite[Theorem 1.11]{Bóna}. Further properties are presented in \cite[Section 1.1]{Bóna} and \cite[Section 6.5]{Comtet}.

The comparison of \eqref{eq:Euler-explicit} and \eqref{eq:Jnr-explicit} implies that
\begin{equation}\label{eq:Euler=Laplace--Pólya}
    A(m,l)=m!J_{m+1}(2l-m-1)
\end{equation}
for arbitrary nonnegative integers $m$ and $l$. Hence, our results for the Laplace--Pólya integral translate to estimates for Eulerian numbers via the connection
\[\dfrac{A(m,l+1)}{A(m,l)}=\frac{J_{m+1}(2l-m+1)}{J_{m+1}(2l-m-1)}.\]
This allows us to derive lower and upper bounds on the ratio of consecutive Eulerian numbers from Theorems~\ref{th:upper} and~\ref{th:lower}:

\begin{corollary}\label{cor:Eulerian-ratio}
    For integers $m\geq3$ and $\frac{m}{2}\leq l\leq m-1$,
	    \begin{equation}\label{eq:Ratio-Eulerian}
	        c_{m+1,2l-m-1}\leq\frac{A(m,l+1)}{A(m,l)}\leq d_{m+1,2l-m-1},
	    \end{equation}
	    where $c_{n,r}$ and $d_{n,r}$ are defined by formulae \eqref{eq:cnr} and {\eqref{eq:dnr}} respectively.
\end{corollary}

The upper bound is from \cite[Proposition 3.1]{Ambrus-Gárgyán}. Note that the lower bound becomes negative and thus trivial for $\frac12\big\lfloor\frac{11m+3}{7}\big\rfloor <l \leq m-1$. Furthermore, due to the symmetry relation \eqref{eq:Euler-symm}, analogous estimates hold for $0 \leq l < \frac m 2$.

Eulerian numbers were studied extensively by Lesieur and Nicolas~\cite{Lesieur-Nicolas}. Following their work, we introduce the notation
\begin{equation}\label{eq:Eulerian-Max}
    M_m:=\max_{l=1,\ldots,m}A(m,l)=A\Big(m,\Big\lfloor\frac m2\Big\rfloor+1\Big)
\end{equation}
for the central Eulerian numbers in a row.

According to~\eqref{eq:Euler=Laplace--Pólya}, 
\begin{equation}\label{eq:M_m=Laplace--Pólya}
    {M_m}=m!J_{m+1}\Big(2\Big\lfloor\frac m2\Big\rfloor-m+1\Big)=\begin{cases}m!J_{m+1}(1),&\text{if}\ m\ \text{is even,}\\m!J_{m+1}(0),&\text{if}\ m\ \text{is odd}.\end{cases}
\end{equation}
As in~\cite[Subsection 3.5]{Lesieur-Nicolas}, we define the function $f(m)$ for $m \geq 1$ by
\begin{equation}\label{eq:fm}
    f(m):=\frac{M_m}{m!}=\begin{cases}J_{m+1}(1),&\text{if}\ m\ \text{is even,}\\J_{m+1}(0),&\text{if}\ m\ \text{is odd},\end{cases}
\end{equation} 
The initial values of the sequence $f(m)_{m=1}^\infty$ are listed in Table~\ref{tab:fm}. 
 \begin{table}[h!] 
		\[\begin{NiceArray}{*{8}{c}}[hvlines-except-borders,cell-space-limits=5pt,columns-width =40 pt]
			m&1&2&3&4&5&6&7\\
			f(m)&1&\dfrac{1}{2}&\dfrac{3}{4}&\dfrac{23}{48}&\dfrac{115}{192}&\dfrac{841}{1920}&\dfrac{5887}{11520}
		\end{NiceArray}\]
		\caption{Values of  $f(m)$ for $m=1,\ldots,7$.}
		\label{tab:fm}
	\end{table}

Using analytic arguments involving estimates of power series coefficients, Lesieur and Nicolas proved the following core result of~\cite{Lesieur-Nicolas}.
\begin{theorem}[{\cite[p. 395, Theorem 2]{Lesieur-Nicolas}}]\label{th:LN}
    For all $p\geq1$,
    \begin{equation}\label{eq:LN-f-ratio}
        \frac{2p}{2p+1}<\frac{f(2p+1)}{f(2p-1)}<\frac{2p+1}{2p+2}.
    \end{equation}
\end{theorem}
Note that via \eqref{eq:fm}, Corollary~\ref{cor:bounds-central} implies the slightly weaker bounds
    \begin{equation}\label{eq:f-ratio-odd}
        \frac{2p}{2p+1}<\frac{f(2p+1)}{f(2p-1)}\leq\frac{(p+1)(2p^2+2p-1)}{p(p+2)(2p+1)}
    \end{equation}
for every $p \geq 2$. 

The above estimates can simply be extended to even indices via the identity
\begin{equation}\label{eq:f-even-odd}
    f(2p)=\frac{2p+1}{2p+2}f(2p+1)
\end{equation}
for any $p\geq1$ that is a consequence of \eqref{eq:fm} and \eqref{eq:Jn0}. Thus, \eqref{eq:f-ratio-odd} leads to 
    \begin{equation}\label{eq:f-ratio-even}
        \frac{2p^2}{2p^2+p-1}<\frac{f(2p)}{f(2p-2)}\leq\frac{2p^2+2p-1}{(p+2)(2p-1)}.
    \end{equation}
which is slightly weaker than 
    \begin{equation*}
        \frac{2p^2}{2p^2+p-1}<\frac{f(2p)}{f(2p-2)}<\frac{p (2 p +1)^2}{2 (p+1)^2(2p-1)}.
    \end{equation*}
implied by Theorem~\ref{th:LN}.

\begin{figure}[h]
    \centering
    \includegraphics[scale=.43]{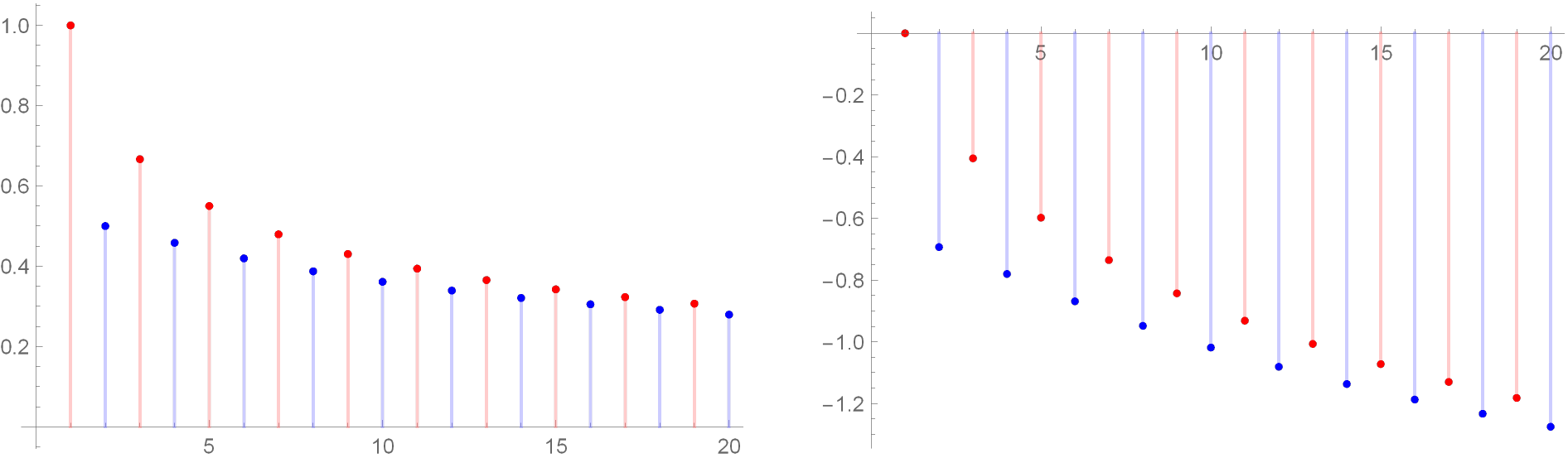}
    \caption{The sequence $f(m)$ and its logarithm up to $m=20$, marked with different colors for odd and even values of $m$.}
    \label{fig:fm}
\end{figure}

Note that \eqref{eq:f-ratio-odd} and \eqref{eq:f-even-odd} show that the sequences $(f(2p))_{p=1}^\infty$ and $(f(2p+1)))_{p=0}^\infty$ are decreasing, which was proved in \cite[p. 396, Theorem 3]{Lesieur-Nicolas}. We conclude this section by establishing the relevant convexity properties; see Figure~\ref{fig:fm}. 

\begin{proposition}\label{prop:f-odd}
    The sequence $\big(f(2p+1)\big)_{p=0}^\infty$ is convex and logarithmically convex.
\end{proposition}

\begin{proposition}\label{prop:f-even}
    The sequence $\big(f(2p)\big)_{p=2}^\infty$ is convex.
\end{proposition}

\begin{proof}[Proof of Proposition~\ref{prop:f-odd}]
For the initial segment of the sequence $\big(f(2p+1)\big)_{p=0}^3$, the statement is verified by direct calculation based on Table~\ref{tab:fm}.

For $p \geq 2$, proving convexity
and log-convexity amounts to showing that
\begin{equation*}
    \frac{f(2p+3)}{f(2p+1)}+\frac{f(2p-1)}{f(2p+1)}\geq2\quad\text{and}\quad\frac{f(2p+3)}{f(2p+1)}\cdot\frac{f(2p-1)}{f(2p+1)}\geq1,
\end{equation*}
respectively. Applying the bounds in \eqref{eq:f-ratio-odd} leads to
\begin{align*}
    \frac12\Big(\frac{f(2p+3)}{f(2p+1)}+\frac{f(2p-1)}{f(2p+1)}\Big)&\geq
    \frac{1}{2}\Big(\frac{2p+2}{2p+3}+\frac{p(p+2)(2p+1)}{(p+1)(2p^2+2p-1)}\Big)=\\[5pt]&=
    \frac{8p^4+28p^3+29p^2+6p-2}{8p^4+28p^3+28p^2+2p-6}>1,
\end{align*}
furthermore
\begin{align*}
    \frac{f(2p+3)}{f(2p+1)}\cdot\frac{f(2p-1)}{f(2p+1)}&\geq
    \frac{2p+2}{2p+3}\cdot\frac{p(p+2)(2p+1)}{(p+1)(2p^2+2p-1)}=
    \\[5pt]&=
    \frac{4p^3+10p^2+4p}{4p^3+10p^2+4p-3}>1.
    \qedhere
\end{align*}
\end{proof}

\begin{proof}[Proof of Proposition~\ref{prop:f-even}]
Convexity is implied by applying both the lower and upper estimates of \eqref{eq:f-ratio-even}:
\begin{align*}
    \frac12\Big(\frac{f(2p+2)}{f(2p)}+\frac{f(2p-2)}{f(2p)}\Big)&\geq
    \frac{1}{2}\Big(\frac{2p^2+4p+2}{2p^2+5p+2}+\frac{(p+2)(2p-1)}{2p^2+2p-1}\Big)=\\[5pt]&=
    \frac{8p^4+28p^3+25p^2-4p-6}{8p^4+28p^3+24p^2-2p-4}>1.
    \qedhere
\end{align*}
\end{proof}

We conjecture that the sequence $\big(f(2p)\big)_{p=2}^\infty$ is log-convex as well, although \eqref{eq:f-ratio-even} leads only to the slightly weaker inequality
\begin{align*}
    \frac{f(2p+2)}{f(2p)}\cdot\frac{f(2p-2)}{f(2p)}&\geq
        \frac{4p^3+6p^2-2}{4p^3+6p^2-1}\,.
\end{align*}

\smallskip 

\noindent
{\bf Acknowledgments.} The authors are grateful to Oscar Ortega Moreno for insightful discussions and for proposing the study of local extremality of diagonal cube sections, to Lionel Pournin for his support and collaboration, and to the anonymous referees whose comments greatly enhanced the clarity of the presentation.

\bibliographystyle{amsplain}
        \bibliography{LaplacePolya_references}
        
\medskip     

\noindent
{\sc Gergely Ambrus}

\noindent
{\em Bolyai Institute, University of Szeged, Hungary, \\ and HUN-REN Alfréd Rényi Institute of Mathematics,  Budapest, Hungary}

\noindent
e-mail address: \texttt{ambrus@renyi.hu}

\medskip

\noindent
{\sc Barnabás Gárgyán}

\noindent
{\em  Mathematics Institute, University of Warwick, Coventry, United Kingdom, \\ and Bolyai Institute, University of Szeged, Hungary}

\noindent
e-mail address: \texttt{Barnabas.Gargyan@warwick.ac.uk}
\end{document}